\documentclass{amsart}

\usepackage{amssymb}
\usepackage{amsmath}
\usepackage{amsthm}
\usepackage{latexsym}
\usepackage{amsfonts}
\usepackage{mathrsfs}
\usepackage{setspace}
\usepackage{fancyhdr}

\newtheorem{thm}{Theorem}[section]
\newtheorem{cor}[thm]{Corollary}
\newtheorem{lem}[thm]{Lemma}
\newtheorem{prop}[thm]{Proposition}

\theoremstyle{definition}

\newtheorem{rem}[thm]{Remark}
\newtheorem{ex}[thm]{Example}

\numberwithin{equation}{section}

\newcommand{\R}{\ensuremath{\mathbb R}}    
\newcommand{\C}{\ensuremath{\mathbb C}}    
\newcommand{\N}{\ensuremath{\mathbb N}}    

\newcommand{\product}{[\cdot\,,\cdot]}

\newcommand{\calD}{\mathcal D}

\newcommand{\calS}{\mathcal S}

\newcommand{\la}{\lambda}
\newcommand{\veps}{\varepsilon}
\newcommand{\vphi}{\varphi}

\newcommand{\bmat}{\begin{pmatrix}}
\newcommand{\emat}{\end{pmatrix}}
\newcommand{\mat}[4]
{
   \begin{pmatrix}
      #1 & #2\\
      #3 & #4
   \end{pmatrix}
}
\newcommand{\vek}[2]
{
   \begin{pmatrix}
      #1\\
      #2
   \end{pmatrix}
}

\renewcommand{\Im}{\operatorname{Im}}
\renewcommand{\Re}{\operatorname{Re}}

\newcommand{\dom}{\operatorname{dom}}

\newcommand{\ol}{\overline}

\newcommand{\AC}{\operatorname{AC}}

\newcommand{\sgn}{\operatorname{sgn}}

\begin{document}
\title[Non-real eigenvalues of indefinite Sturm-Liouville problems]{Estimates on the non-real eigenvalues of
regular indefinite Sturm-Liouville problems}

\author{Jussi Behrndt}
\address{Technische Universit\"{a}t Graz, Institut f\"{u}r Numerische Mathematik, Steyrergasse 30, 8010 Graz, Austria}
\email{behrndt@tugraz.at}
\urladdr{www.math.tugraz.at/~behrndt/}

\author{Shaozhu Chen}
\address{Department of Mathematics, Shandong University (Weihai), Weihai 264209, P.R., China}
\email{szchen@sdu.edu.cn}
\urladdr{}

\author{Friedrich Philipp}
\address{Institut f\"ur Mathematik, Technische Universit\"at Clausthal, Erzstra\ss e 1, 38678 Clausthal-Zellerfeld, Germany}
\email{fmphilipp@gmail.com}
\urladdr{www.tu-clausthal.de/~fph12}

\author{Jiangang Qi}
\address{Department of Mathematics, Shandong University (Weihai), Weihai 264209, P.R., China}
\email{qjg816@163.com}
\urladdr{}

\begin{abstract}
Regular Sturm-Liouville problems with indefinite weight functions may possess finitely many non-real eigenvalues.
In this note we prove explicit bounds on the real and imaginary parts of these eigenvalues in terms
of the coefficients of the differential expression.

\end{abstract}



\maketitle

\section{Introduction}

In this paper we consider regular {\it indefinite} Sturm-Liouville eigenvalue problems of the form
\begin{equation}\label{e:SLEq}
\tau(f)=\lambda f\quad\text{with}\quad \tau=\frac{1}{w}\left(-\frac{d}{dx}\,p\,\frac{d}{dx}+q\right)
\end{equation}
on bounded intervals $(a,b)\subset\R$ with real coefficients $p^{-1},q,w\in L^1(a,b)$ such that $p> 0$ and $w\not=0$ a.e.\ on $(a,b)$.
The problem \eqref{e:SLEq} is supplemented with suitable boundary conditions at the endpoints $a$ and $b$.
The pecularity here is that the weight function $w$ is not assumed to be positive and for this reason the eigenvalue
problem and the Sturm-Liouville differential expression $\tau$ in \eqref{e:SLEq} is called {\it indefinite}.

The history of indefinite Sturm-Liouville eigenvalue problems goes back to the early 20th century, where Haupt \cite{h}
and Richardson \cite{r} generalized oscillation results to the indefinite case, and noted that problems of the form \eqref{e:SLEq}
may have non-real eigenvalues. For more historical details and
other classical references we refer the reader to the interesting survey paper \cite{m2} by Mingarelli.
From a modern and more abstract point of view the spectral theory of Sturm-Liouville operators with indefinite weights is intimately
connected with spectral and perturbation theory of operators which are selfadjoint with respect to the indefinite inner product
\begin{equation}\label{novisad}
[f,g]:=\int_a^b f(x)\,\overline{g(x)}\, w(x)\, dx,
\end{equation}
where $f,g$ are functions in the weighted $L^2$-space $L^2_{\vert w\vert}(a,b)$.
The qualitative spectral properties of the selfadjoint differential operators associated to $\tau$ in the
Krein space $(L^2_{\vert w\vert}(a,b),\product)$ are well understood. We emphasize the
contribution \cite{cl} by \'Curgus and Langer in which many operator theoretic fundaments of the theory were laid. 
In particular, the spectrum of any selfadjoint realization consists of normal eigenvalues only.
There are at most finitely many non-real eigenvalues which appear in pairs symmetric with respect to the real axis, and
the real eigenvalues accumulate to $+\infty$ and $-\infty$. We refer to the monograph \cite{z} by Zettl for an overview and to
\cite{B85,BT07,BBW02,BLM04,BM10,BV01,KKM09,p} for some other aspects in indefinite Sturm-Liouville theory.

The main objective of this paper is to prove bounds on the non-real spectrum of indefinite Sturm-Liouville operators in
terms of the coefficients in the differential expression.
This is a challenging open problem according to Mingarelli~\cite{m2} and Kong, M\"{o}ller, Wu, and Zettl \cite[Remark 4.4]{KMWZ03}, see also 
\cite[Remark 11.4.1]{z}. 
Only in the very recent past first results in this direction were obtained by the authors of
this paper independently in \cite{bpt} jointly with Trunk for a particular singular problem, and in \cite{qc} for regular problems with Dirichlet
boundary conditions, special weight functions $w$ and $p=1$.

In this note we investigate the general regular case with arbitrary selfadjoint boundary conditions.
The only restriction on the weight $w$
is that we assume the existence of an absolutely continu\-ous function $g$ with $g^{\prime\, 2} p\in L^1(a,b)$ such that $\sgn(g) = \sgn(w)$ a.e.
In Theorem~\ref{t:main} and Theorem~\ref{t:main2} we then obtain bounds for the real and imaginary parts of the non-real eigenvalues
of the indefinite Sturm-Liouville eigenvalue
problem \eqref{e:SLEq} which depend on $p$, $q$, $g$ (and thus implicitly on $w$), and the selfadjoint boundary condition.
The techniques in the proofs of our main results are inspired by the methods in \cite{qc}.
For the case of a weight function with finitely many sign changes we construct an admissible function $g$ and find
bounds which do not depend on $g$ in Corollaries~\ref{c:ftp} and \ref{c:ftp2}. A particular weight function with infinitely many turning points
is treated in Example~\ref{auweia}.
Furthermore, for a certain set of real eigenvalues where the eigenfunctions have special sign properties (sometimes called real ghost states)
we obtain similar bounds as in Theorem~\ref{t:main} in Theorem~\ref{t:main3}.

The paper is organized as follows. After introducing the relevant notions in Section 2,
we prove the a priori bounds on the non-real eigenvalues of indefinite regular Sturm-Liouville operators in Section 3.
Section 4 contains the estimates on the real exceptional eigenvalues. A key ingredient in the proofs of
the results in Sections 3 and 4 are certain estimates on the norms of the corresponding eigenfunctions and their derivatives
in Lemmas \ref{lemest1}, \ref{lemest2}, and \ref{lemest3}. In order to improve the reading flow we
outsourced the proofs of these lemmas into the separate Section 5.

\section{Preliminaries}\label{s:always}

Let $\tau$ be the indefinite Sturm-Liouville expression from \eqref{e:SLEq} with real-valued coefficients
$p^{-1},q,w\in L^1(a,b)$ such that $p> 0$ and $w\not=0$ a.e.\ on $(a,b)$.
It will be assumed that both sets
$$\bigl\{x\in (a,b):w(x)>0\bigr\}\quad\text{and}\quad\bigl\{x\in (a,b):w(x)<0\bigr\}$$
have positive Lebesgue measure.
Let $L^2_{\vert w\vert}(a,b)$ be the linear space (of equivalence classes) of measurable functions $f:(a,b)\rightarrow\C$ such that
$f^2 w\in L^1(a,b)$ and equip this space with the indefinite inner product $\product$ in \eqref{novisad}.

The differential expression $\tau$ is then formally symmetric with respect to $\product$ and hence gives rise to selfadjoint realizations
in the Krein space $(L^2_{\vert w\vert}(a,b),\product)$, that is, $\tau$ induces a family of differential operators which are selfadjoint
with respect to the Krein space inner product $\product$. In the remainder of this paper {\it selfadjoint} refers to selfadjointness
with respect to this inner product.

Let us briefly recall how the selfadjoint realizations of $\tau$ can be para\-me\-trized; cf. \cite[Section 4.2]{z}. For this denote by $\calD_{\max}$ the maximal domain which consists of all $f\in L^2_{|w|}(a,b)$ such that $f,pf'$
are absolutely continuous and $\tau(f)\in L^2_{|w|}(a,b)$. Then any selfadjoint differential operator associated to $\tau$ in $(L^2_{\vert w\vert}(a,b),\product)$ is of the form
\begin{equation}\label{e:diffop}
A(\calD)f=\tau (f),\qquad \dom A(\calD) = \calD,
\end{equation}
where
$$
\calD = \calD_{\rm sep}(l,r) := \bigl\{f\in\calD_{\max} : (pf')(a) = lf(a),\;(pf')(b) = rf(b)\bigr\}
$$
with $l,r\in \R\cup\{\infty\}$ or
$$
\calD = \calD_{\rm coup}(\vphi,R) := \left\{f\in\calD_{\max} : \vek{f(b)}{(pf')(b)} = e^{i\vphi}R\vek{f(a)}{(pf')(a)}\right\}
$$
with $\vphi\in [0,2\pi)$ and $R\in\R^{2\times 2}$ such that $\det R = 1$. We note that $l=\infty$ or $r=\infty$ in
$\calD_{\rm sep}(l,r)$ stands for the Dirichlet boundary condition at $a$ or $b$, respectively. For brevity we shall refer to the above domains as {\it selfadjoint domains}. To any selfadjoint domain $\calD$ we assign a constant $c(\calD)\ge 0$ as follows:
\begin{equation}\label{e:cD}
c(\calD) :=
\begin{cases}
|l| + |r|             &\text{if $\calD = \calD_{\rm sep}(l,r)$ with $l,r\in\R$},\\
|r|                   &\text{if $\calD = \calD_{\rm sep}(\infty,r)$ with $r\in\R$},\\
|l|                   &\text{if $\calD = \calD_{\rm sep}(l,\infty)$ with $l\in\R$},\\
0                     &\text{if $\calD = \calD_{\rm sep}(\infty,\infty)$},\\
\frac{|r_{11}| + |r_{22}| + 2}{|r_{12}|}
                      &\text{if $\calD = \calD_{\rm coup}(\vphi,R)$ and $r_{12}\neq 0$},\\
|r_{11}r_{21}|        &\text{if $\calD = \calD_{\rm coup}(\vphi,R)$ and $r_{12} = 0$},
\end{cases}
\end{equation}
where the $r_{ij}$'s are the entries of the matrix $R = (r_{ij})_{i,j=1}^2\in\R^{2\times 2}$ in the case of coupled boundary conditions, i.e. $\calD = \calD_{\rm coup}(\vphi,R)$.

\begin{lem}\label{l:D}
Let $\calD$ be a selfadjoint domain and let $\phi\in\calD$. Then we have
\begin{equation}\label{e:sa}
\Im\big((p\phi')(b)\ol{\phi(b)} - (p\phi')(a)\ol{\phi(a)}\,\big) = 0,
\end{equation}
and, in addition,
\begin{equation}\label{e:wichtig}
\big|(p\phi')(b)\ol{\phi(b)} - (p\phi')(a)\ol{\phi(a)}\big|\le c(\calD)\max\bigl\{|\phi(a)|^2,|\phi(b)|^2\bigr\}.
\end{equation}
\end{lem}

\begin{proof}
The identity \eqref{e:sa} follows from the selfadjointness of $A(\calD)$.
We only show that \eqref{e:wichtig} holds in the case $\calD = \calD_{\rm coup}(\vphi,R)$. The other cases are evident.
Let $\phi\in\calD_{\rm coup}(\vphi,R)$. Then we have
\begin{equation}\label{e:eins}
\vek{\phi(b)}{(p\phi')(b)} = e^{i\vphi}\mat{r_{11}}{r_{12}}{r_{21}}{r_{22}}\vek{\phi(a)}{(p\phi')(a)},
\end{equation}
and hence, as $\det R = r_{11}r_{22} - r_{12}r_{21} = 1$, also
\begin{equation}\label{e:zwei}
\mat{r_{22}}{-r_{12}}{-r_{21}}{r_{11}}\vek{\phi(b)}{(p\phi')(b)} = e^{i\vphi}\vek{\phi(a)}{(p\phi')(a)}.
\end{equation}
From \eqref{e:eins} we get
$$
r_{12}(p\phi')(a) = e^{-i\vphi}\phi(b) - r_{11}\phi(a),
$$
and \eqref{e:zwei} yields
$$
r_{12}(p\phi')(b) = r_{22}\phi(b) - e^{i\vphi}\phi(a).
$$
Hence, if $r_{12}\neq 0$ then
$$
(p\phi')(b)\ol{\phi(b)} - (p\phi')(a)\ol{\phi(a)}
= \frac{r_{22}|\phi(b)|^2 + r_{11}|\phi(a)|^2 - 2\Re\big(e^{i\vphi}\phi(a)\ol{\phi(b)}\big)}{r_{12}}.
$$
This directly implies \eqref{e:wichtig}. If $r_{12} = 0$, then, first of all, $r_{11}r_{22} = 1$. Moreover, by \eqref{e:zwei} we have $r_{22}\phi(b) = e^{i\vphi}\phi(a)$, and from \eqref{e:eins} we get
$$
e^{-i\vphi}(p\phi')(b) = r_{21}\phi(a) + r_{22}(p\phi')(a).
$$
This yields
\begin{align*}
(p\phi')(b)\ol{\phi(b)} - (p\phi')(a)\ol{\phi(a)}
&= \big(r_{22}^{-1}(p\phi')(b)e^{-i\vphi} - (p\phi')(a)\big)\ol{\phi(a)}\\
&= \left(r_{22}^{-1}\left(r_{21}\phi(a) + r_{22}(p\phi')(a)\right) - (p\phi')(a)\right)\ol{\phi(a)}\\
&= r_{11}r_{21}|\phi(a)|^2,
\end{align*}
and \eqref{e:wichtig} follows.
\end{proof}

For the estimates on the non-real and exceptional eigenvalues in the next sections we need a set of norms.
If $r : (a,b)\to [0,\infty)$ is a measurable function we denote by $\mu_r$ the measure on $(a,b)$ with $d\mu_r = r\,dt$ and
define the weighted $L^2$-spaces as $L^2_r(a,b) := L^2((a,b),\mu_r)$; this is in accordance with $L^2_{\vert w\vert}(a,b)$ defined above.
The norm of $L^2_r(a,b)$ will be denoted by $\|\cdot\|_{r,2}$. As usual the $L^1$-norm and $L^\infty$-norm will be denoted by
$\|\cdot\|_1$ and $\|\cdot\|_\infty$, respectively.

We close this section with a simple observation which will be exploited in many of the proofs below. Let $\phi$ be a solution of the equation \eqref{e:SLEq}, i.e., $\phi,p\phi'\in\AC[a,b]$ and
\begin{equation}\label{ode}
-(p\phi^\prime)^\prime+q\phi=\lambda w\phi.
\end{equation}
Multiplying \eqref{ode} with $\ol\phi$ and using  $(p\phi'\ol\phi)' = (p\phi')'\ol\phi + p|\phi'|^2$ we obtain
\begin{equation}
\la w|\phi|^2 = -(p\phi^\prime)^\prime\ol\phi+q\vert \phi\vert^2 = - (p\phi'\ol\phi)' + p|\phi'|^2 + q|\phi|^2.
\end{equation}
Integration over $[x,b]\subset[a,b]$ gives
\begin{equation}\label{e:integration}
\la \int_x^b w|\phi|^2 = (p\phi')(x)\ol{\phi(x)} - (p\phi')(b)\ol{\phi(b)} + \int_x^b\big(p|\phi'|^2 + q|\phi|^2\big)
\end{equation}
and for the real and imaginary part we conclude
\begin{equation}\label{e:real}
(\Re\la)\int_x^b w|\phi|^2 = \Re\bigl((p\phi')(x)\ol{\phi(x)} - (p\phi')(b)\ol{\phi(b)}\,\bigr) + \int_x^b\big(p|\phi'|^2 + q|\phi|^2\big)
\end{equation}
and
\begin{equation}\label{e:im}
(\Im\la)\int_x^b w|\phi|^2 = \Im\bigl((p\phi')(x)\ol{\phi(x)} - (p\phi')(b)\ol{\phi(b)}\,\bigr).
\end{equation}

\vspace{.5cm}
\section{Bounds on non-real eigenvalues}
In this section we provide a priori bounds on the non-real eigenvalues of the selfadjoint realizations of the regular indefinite
Sturm-Liouville expression $\tau$; cf. Theorem~\ref{t:main} and Theorem~\ref{t:main2} below.
The following constants will be incorporated into these bounds.
\begin{equation}\label{ab}
\alpha:=c(\calD)+\Vert q_-\Vert_1,\;\;\beta:=\sqrt{\alpha\bigl(1 / \Vert p^{-1}\Vert_1 +\alpha\bigr)}+\alpha,\;\;
\gamma:=\sqrt{2\beta + 1/\Vert p^{-1}\Vert_1}.
\end{equation}
Here (and in the following), $q_-(x) := \min\{0,q(x)\}$, $x\in (a,b)$. Note that $\alpha$, $\beta$, and $\gamma$ only
depend on the chosen selfadjoint boundary conditions and the norms $\Vert q_-\Vert_1$, $\Vert p^{-1}\Vert_1$.
In particular, the constants $\alpha$, $\beta$, and $\gamma$ do not depend on the weight function $w$.

The following lemma is the first of three similar statements
which play a key role in the proofs of the eigenvalue estimates in this paper. Its proof can be found in Section~\ref{s:lemmas}.

\begin{lem}\label{lemest1}
Let $\calD$ be a selfadjoint domain. Then for all $\lambda\in\C\setminus\R$ and all solutions $\phi\in\calD$ of the equation {\rm\eqref{ode}} the following estimates hold:
\begin{equation*}
  \Vert\phi^\prime\Vert_{p,2}\leq \beta \Vert\phi\Vert_{\frac{1}{p},2}\qquad\text{and}\qquad\Vert\phi\Vert_\infty\leq \gamma\,
\Vert\phi\Vert_{\frac{1}{p},2}.
\end{equation*}
\end{lem}

The next theorem is the first main result of this section. It provides estimates on the real and
imaginary parts of the non-real eigenvalues of the operators $A(\calD)$ in \eqref{e:diffop} for any selfadjoint domain $\calD$.

\begin{thm}\label{t:main}
Let $\calD$ be a selfadjoint domain, let $\alpha$, $\beta$, $\gamma$ be as in \eqref{ab}, and
assume that there exists a real-valued function $g\in\AC[a,b]$ with $g(a) = g(b) = 0$ and $g'\in L^2_p(a,b)$ such that $gw > 0$ a.e.\ on $(a,b)$. Then, with $\veps > 0$ chosen such that
\begin{equation}\label{e:condi}
\mu_{\frac{1}{p}}\big(\{x\in [a,b] : p(x)g(x)w(x)<\veps\}\big)\,\le\,\frac{1}{2\gamma^2}\,,
\end{equation}
the following holds for all eigenvalues $\lambda\in\C\setminus\R$ of the operator $A(\calD)$:
\begin{equation}\label{e:im_est}
|\Im\la\,|\,\le\,\frac{2}{\veps}\,\beta\,\gamma\,\|g'\|_{p,2}
\end{equation}
and
\begin{equation}\label{e:re_est}
|\Re\la\,|\,\le\,\frac{2}{\veps}\,\bigl(\beta\,\gamma\,\|g'\|_{p,2} + (\beta^2 + \gamma^2\|q\|_1)\|g\|_\infty\bigr).
\end{equation}
\end{thm}
\begin{proof}
Let $\lambda\in\C\setminus\R$ be a non-real eigenvalue of $A(\calD)$ and let $\phi\in\calD$ be a corresponding eigenfunction. It is no restriction to assume that
\begin{equation}\label{norm1}
 \Vert\phi\Vert_{\frac{1}{p},2} = 1.
\end{equation}
Since $g(a) = g(b) = 0$, integration by parts yields
\begin{equation}\label{gutenmorgen}
\begin{split}
 \int_a^b g'(x)\int_x^b w(t)|\phi(t)|^2\,dt\,dx&=-\int_a^b g(x)\,\frac{d}{dx}\int_x^b w(t)|\phi(t)|^2\,dt\,dx \\ &= \int_a^b g(x)w(x)|\phi(x)|^2\,dx.
\end{split}
\end{equation}
Let $\Omega := \{x\in [a,b] : p(x)g(x)w(x) < \veps\}$ and $\Omega^c = [a,b]\setminus\Omega$. By assumption we
have $\mu_{p^{-1}}(\Omega)\le\frac 1{2\gamma^2}$
and hence we find
\begin{align}
\begin{split}\label{e:meins}
\int_a^b gw|\phi|^2 &= \int_a^b (pgw)(|\phi|^2 p^{-1})\ge\,\veps\int_{\Omega^c}|\phi|^2 p^{-1}= \veps\left(1 - \int_\Omega|\phi|^2 p^{-1}\right)\\
&\ge\,\veps\left(1 - \Vert\phi\Vert_\infty^2\mu_{p^{-1}}(\Omega)\right)\,\ge\,\veps\left(1 - \gamma^2\mu_{p^{-1}}(\Omega)\right)\,\ge\,\frac\veps 2,
\end{split}
\end{align}
where we have used the estimate $\Vert\phi\Vert_\infty^2\leq 
\gamma^2$ from Lemma~\ref{lemest1} together with \eqref{norm1}. Combining \eqref{gutenmorgen} and \eqref{e:meins}
we have shown the estimate
\begin{equation}\label{nutzlich}
 \frac\veps 2\leq\int_a^b g'(x)\int_x^b w(t)|\phi(t)|^2\,dt\,dx.
\end{equation}
From this, \eqref{e:im} and $g(a)=g(b)=0$ we obtain
\begin{align*}
|\Im\la\,|\frac\veps 2
&\le\left|\int_a^b g'(x)\,(\Im\lambda)\int_x^b w(t)|\phi(t)|^2\,dt\,dx\right|\\
&=\left|\int_a^b g'(x)\Im\bigl((p\phi')(x)\ol{\phi(x)} - (p\phi')(b)\ol{\phi(b)}\,\bigr)\,dx\right|\\
&=\left|\int_a^b g'(x)\Im \bigl((p\phi')(x)\ol{\phi(x)}\,\bigr)\,dx\right|
\end{align*}
which can be estimated further as follows
\begin{equation*}
\le\,\int_a^b|g'p\phi'\phi|\,\le\,\Vert\phi\Vert_\infty\int_a^b|g'\vert p^{1/2} \,\vert\phi'|p^{1/2}\,\le\,\Vert\phi\Vert_\infty
\|g'\|_{p,2}\|\phi'\|_{p,2}.
\end{equation*}
Thus the assertion on the imaginary part of the eigenvalue $\lambda$ follows from the estimates
$\|\phi'\|_{p,2}\leq\beta$ and $\Vert\phi\Vert_\infty\leq\gamma$;
cf. Lemma~\ref{lemest1} and \eqref{norm1}.

It remains to estimate the real part of $\lambda$. For this we set $G(x):=\int_x^b p\vert\phi^\prime\vert^2+q\vert\phi\vert^2$.
From \eqref{nutzlich} and \eqref{e:real} we have
\begin{equation}\label{nabitte}
\begin{split}
 |\Re\la\,|\frac\veps 2 & \leq\left|\int_a^b g'(x)\left(\Re\bigl((p\phi')(x)\ol{\phi(x)} - (p\phi')(b)\ol{\phi(b)}\,\bigr) + G(x)\right)\,dx\right|\\
&=\left|\int_a^b g'(x)\, \Re\bigl( (p\phi')(x)\ol{\phi(x)}\,\bigr)\,dx + \int_a^b g'(x) G(x)\,dx \right|.
\end{split}
\end{equation}
Integration by parts shows
$$
\int_a^b g'(x) G(x)\,dx=-\int_a^b g(x)G^\prime(x)\,dx=\int_a^b g(x) \bigl(p(x)\vert\phi^\prime(x)\vert^2+q(x)\vert\phi(x)\vert^2\bigr)\,dx
$$
and hence from \eqref{nabitte} we have 
\begin{equation}\label{nabitte2}
\begin{split}
 |\Re\la\,|\frac\veps 2 &\leq \int_a^b \vert g^\prime p\phi^\prime\phi\vert + \int_a^b \left|g\bigl(p\vert\phi^\prime\vert^2+q\vert\phi\vert^2\bigr)\right|\\
 &\leq\Vert\phi\Vert_\infty\int_a^b\vert g^\prime\vert p^{1/2} \,\vert\phi^\prime\vert p^{1/2} + \Vert g\Vert_\infty \int_a^b \left|p\vert\phi^\prime\vert^2+q\vert\phi\vert^2\right|\\
 &\leq\Vert\phi\Vert_\infty\Vert g^\prime\Vert_{p,2}\Vert \phi^\prime\Vert_{p,2}+ \Vert g\Vert_\infty\bigl( \Vert\phi^\prime\Vert_{p,2}^2+
\Vert q\Vert_1\Vert \phi\Vert_\infty^2\bigr).
\end{split}
\end{equation}
The inequality \eqref{nabitte2} together with the estimates
$\|\phi'\|_{p,2}\leq\beta$ and $\Vert\phi\Vert_\infty\leq\gamma$ (see Lemma~\ref{lemest1} and \eqref{norm1}) imply
$$
|\Re\la\,|\frac\veps 2\leq\beta\,\gamma\,\Vert g^\prime\Vert_{p,2} + \Vert g\Vert_\infty\bigl( \beta^2+ \gamma^2 \Vert q\Vert_1\bigr)
$$
which is \eqref{e:re_est}. The theorem is proved.
\end{proof}

As the condition in Theorem \ref{t:main} concerning the existence of the absolutely continuous function $g$ is somewhat implicit, we show in the next corollary how the theorem becomes more explicit in the case of an indefinite weight function with a finite number of turning points, that is, the interval $(a,b)$ can be segmented into a finite number of intervals on each of which $\sgn(w)$ is constant.

\begin{cor}\label{c:ftp}
Assume $p = 1$ and that $w$ has $n$ turning points in $(a,b)$. Moreover, let $\calD$ be a selfadjoint domain and let $\alpha$, $\beta$, $\gamma$ be as in {\rm\eqref{ab}}. Then, with $\veps > 0$ chosen such that
$$
\mu_1\big(\{x\in (a,b) : |w(x)| < \veps\}\big)\,\le\,\frac 1 {4\gamma^2},
$$
the following holds for all eigenvalues $\la\in\C\setminus\R$ of the operator $A(\calD)$:
$$
|\Im\la\,|\,\le\,\frac{8}{\veps}\,\beta\,\gamma^2\,(n+1)
\quad\text{and}\quad
|\Re\la\,|\,\le\,\frac{2}{\veps}\bigl(4\, \beta\,\gamma^2\,(n+1) + (\beta^2 + \gamma^2\|q\|_1)\bigr).
$$
\end{cor}

\begin{proof}
Let $x_1 < \ldots < x_n$ be the turning points of $w$ in $(a,b)$, put $x_0 := a$, $x_{n+1} := b$, and define the constant
$$
\nu := \frac{1}{8(n+1)\gamma^2}\,.
$$
Let $k\in\{0,\ldots,n\}$. If $x_{k+1} - x_k\ge 2\nu$, for $x\in [x_k,x_{k+1}]$ we set
$$
g(x) := \sgn(w|(x_k,x_{k+1}))\cdot
\begin{cases}
\frac{x - x_k}{\nu}     & \text{for }x\in [x_k,x_k + \nu]\\
1                          & \text{for }x\in [x_k + \nu,x_{k+1} - \nu]\\
\frac{x_{k+1} - x}{\nu} & \text{for }x\in [x_{k+1} - \nu,x_{k+1}].
\end{cases}
$$
If $x_{k+1} - x_k < 2\nu$, then we define
$$
g(x) := \sgn(w|(x_k,x_{k+1}))\cdot\frac{(x - x_k)(x_{k+1} - x)}{2\nu^2},\qquad x\in [x_k,x_{k+1}].
$$
Obviously, we have $g\in\AC[a,b]$, $g'\in L^2(a,b)$, $g(a) = g(b) = 0$, and $gw > 0$ a.e.. Moreover,
$$
\int_{x_k}^{x_{k+1}}|g'(x)|^2\,dx \le
\begin{cases}
\frac 2{\nu} & \text{if }x_{k+1} - x_k\ge 2\nu\\
\frac{x_{k+1} - x_k}{\nu^2} & \text{if }x_{k+1} - x_k < 2\nu
\end{cases}
\;\,\le\;\frac 2{\nu}\,.
$$
In addition, it is easy to see that $|g(x)|\le 1$ for every $x\in (a,b)$. Hence, we obtain
$$
\|g\|_\infty \leq 1\qquad\text{and}\qquad\|g'\|_2\le\sqrt{\frac{2}{\nu}\,(n+1)} = 4\gamma (n+1).
$$
Now, define $\calS := \{x\in (a,b) : |g(x)|\neq 1\}$ and $\Omega := \{x\in (a,b) : g(x)w(x) < \veps\}$. Then we have
$\mu_1(\calS)\leq 2\nu(n+1)$ and hence
$$
\mu_1(\Omega)\,\le\,\mu_1\left(\left\{x\in (a,b)\setminus\calS : |w(x)| < \veps\right\}\right) + 2\nu (n+1)\,\le\,\frac 1 {2\gamma^2}.
$$
The claim now follows from Theorem \ref{t:main}.
\end{proof}

As the following example illustrates, Theorem \ref{t:main} also applies to weight functions with an infinite number of turning points.

\begin{ex}\label{auweia}
Let $p = 1$ and $w(x) = \sin(1/x)$ for $x\in [0,\frac 1 \pi]$. Then $w\in L^1(0,\frac 1 \pi)$, but $w\notin\AC[0,\frac 1 \pi]$ 
(and hence the results in \cite{qc} do not apply here). In order to estimate the non-real eigenvalues of the 
equation \eqref{e:SLEq} with some $q\in L^1(0,\frac 1\pi)$ and some selfadjoint boundary conditions, 
we put $g(x) := x^4\sin(1/x)$. Then $g$ is a function as in Theorem \ref{t:main} with
$$
\|g\|_\infty\le 0.003\qquad\text{and}\qquad\|g'\|_2\le 0.02.
$$
Choose $k_0\in\N$ such that $(k_0+1)\pi > 4\gamma^2$. Then, if for $k\in\N$ we set
$$
I_k := \left[\frac{1}{(k+1)\pi},\frac{1}{k\pi}\right],\qquad\text{we have}\quad\mu_1\left(\bigcup_{k=k_0+1}^\infty I_k\right) < \frac 1 {4\gamma^2},
$$
so that for \eqref{e:condi} to hold it suffices to find $\veps > 0$ such that
\begin{equation}\label{e:mu1}
\mu_1\left(\left\{x\in\bigcup_{k=1}^{k_0} I_k : x^4\sin^2(1/x) < \veps\right\}\right) < \frac 1 {4\gamma^2}.
\end{equation}
For this, we first observe that for $x\in I_k$ we have $x^2|\sin(1/x)|\ge p_k(x)$, where
$$
p_k(x) := \frac 1 \pi(1 - k\pi x)((k+1)\pi x - 1),\qquad x\in I_k.
$$
It is easily seen that for $\veps > 0$ small enough we have
$$
\mu_1\bigl(\{x\in I_k : p_k(x) < \sqrt{\veps}\}\bigr) = \frac{1-\sqrt{1-4\sqrt{\veps}k(k+1)\pi}}{k(k+1)\pi}\,\le\,\frac{4\sqrt{\veps}k(k+1)\pi}{k(k+1)\pi} = 4\sqrt{\veps}.
$$
Hence, with $\sqrt{\veps} := \frac{1}{4k_0(k_0+1)\pi}$ we can estimate the left-hand side of \eqref{e:mu1} by
$$
\sum_{k=1}^{k_0}\mu_1\bigl(\{x\in I_k : x^2|\sin(1/x)|<\sqrt\veps\}\bigr)\,\le\,4k_0\sqrt\veps = \frac{1}{(k_0+1)\pi}\,<\,\frac 1 {4\gamma^2},
$$
so that \eqref{e:mu1}, and hence \eqref{e:condi}, is satisfied. Now, we find estimates on the non-real eigenvalues by making use of \eqref{e:im_est} and \eqref{e:re_est} in Theorem \ref{t:main}.
\end{ex}

In the next lemma we prove estimates different from those in Lemma~\ref{lemest1} for $\Vert\phi^\prime\Vert_{p,2}$ and $\Vert\phi\Vert_\infty$
under the assumption that the weight function $w$ is such that
\begin{equation}\label{lamb}
 \int_a^b w \not=0.
\end{equation}
These involve the constant $\alpha$ in \eqref{ab} and the constant $\delta$ defined by
\begin{equation}\label{deltadelta}
\delta:=2+2\,\frac{\Vert w\Vert_1}{\bigl|\int_a^b w\bigr|}.
\end{equation}
The proof of Lemma \ref{lemest2} can be found in Section 5.

\begin{lem}\label{lemest2}
Assume that the weight function $w$ satisfies \eqref{lamb}.
Then for all $\lambda\in\C\setminus\R$ and all solutions $\phi\in\calD$ of the equation \eqref{ode} the following estimates hold:
\begin{equation*}
  \Vert\phi^\prime\Vert_{p,2}\leq \alpha\,\delta\, \Vert\phi\Vert_{\frac{1}{p},2}\qquad\text{and}\qquad\Vert\phi\Vert_\infty\leq
\sqrt{\alpha}\,\delta\,
\Vert\phi\Vert_{\frac{1}{p},2}.
\end{equation*}
\end{lem}

By the same reasoning as in the proof of Theorem~\ref{t:main} the estimates on $\Vert\phi^\prime\Vert_{p,2}$ and $\Vert\phi\Vert_\infty$ yield bounds on the non-real eigenvalues of the selfadjoint realizations of the regular indefinite Sturm-Liouville expression $\tau$. We note that the estimates in Theorem~\ref{t:main} and Theorem~\ref{t:main2} below are not directly comparable, but can of course be combined if $w$ satisfies assumption~\eqref{lamb}.

\begin{thm}\label{t:main2}
Assume that the weight function $w$ satisfies \eqref{lamb}, let $\calD$ be a selfadjoint domain, let $\alpha$ and $\delta$ be as above and
assume that there exists a real-valued function $g\in\AC[a,b]$ with $g(a) = g(b) = 0$ and $g'\in L^2_p(a,b)$ such that $gw > 0$ a.e.\ on $(a,b)$. Then, with $\veps > 0$ chosen such that
$$
\mu_{\frac{1}{p}}\big(\{x\in [a,b] : p(x)g(x)w(x)<\veps\}\big)\,\le\,\frac{1}{2\alpha\delta^2}\,,
$$
the following holds for all eigenvalues $\lambda\in\C\setminus\R$ of the operator $A(\calD)$:
\begin{equation}\label{e:im_est2}
|\Im\la\,|\,\le\,\frac{2}{\veps}\,\alpha^{3/2}\,\delta^2\,\|g'\|_{p,2}
\end{equation}
and
\begin{equation}\label{e:re_est2}
|\Re\la\,|\,\le\,\frac{2}{\veps}\,\alpha\,\delta^2 \bigl(\sqrt{\alpha}\|g'\|_{p,2}+(\alpha+\Vert q\Vert_1)\Vert g\Vert_\infty \bigr).
\end{equation}
\end{thm}
\begin{proof}
Let $\lambda\in\C\setminus\R$ be an eigenvalue corresponding to some eigenfunction $\phi\in\calD$ and assume that $\phi$
satisfies \eqref{norm1}. The same reasoning as in \eqref{gutenmorgen} and \eqref{e:meins} leads to \eqref{nutzlich},
and hence to the estimates
$$\vert\Im\lambda\,\vert\leq\frac{2}{\veps} \Vert\phi\Vert_\infty
\|g'\|_{p,2}\|\phi'\|_{p,2}
$$
and
$$\vert\Re\lambda\,\vert\leq\frac{2}{\veps}\Bigl(\Vert\phi\Vert_\infty\Vert g^\prime\Vert_{p,2}\Vert \phi^\prime\Vert_{p,2}+ \Vert g\Vert_\infty\bigl( \Vert\phi^\prime\Vert_{p,2}^2+
\Vert q\Vert_1\Vert \phi\Vert_\infty^2\bigr)\Bigr).
$$
Now the assertions follow from $\Vert\phi^\prime\Vert_{p,2}\leq\alpha\, \delta$ and $\Vert\phi\Vert_\infty\leq\sqrt{\alpha}\,\delta$;
cf. Lemma~\ref{lemest2} and \eqref{norm1}.
\end{proof}

The next corollary is a variant of Corollary~\ref{c:ftp} and can be proved in the same way.

\begin{cor}\label{c:ftp2}
Assume $p = 1$, that $w$ satisfies \eqref{lamb} and has $n$ turning points in $(a,b)$. Moreover,
let $\calD$ be a selfadjoint domain and let $\alpha$ and $\delta$ be as in {\rm\eqref{ab}} and {\rm\eqref{deltadelta}}.
Then, with $\veps > 0$ chosen such that
$$
\mu_1\big(\{x\in (a,b) : |w(x)| < \veps\}\big)\,\le\,\frac 1 {4\alpha\delta^2},
$$
the following holds for all eigenvalues $\la\in\C\setminus\R$ of the operator $A(\calD)$:
$$
|\Im\la\,|\,\le\,\frac{8}{\veps}\,\alpha^2\,\delta^3\,(n+1)
\quad\text{and}\quad
|\Re\la\,|\,\le\,\frac{2}{\veps}\,\alpha\,\delta^2 \bigl(4\,\alpha\,\delta\,(n+1) + \alpha + \Vert q\Vert_1\bigr).
$$
\end{cor}

\begin{rem}
If we regard the existence of the function $g$ in Theorems \ref{t:main} and \ref{t:main2} as a condition on the weight function $w$,
it turns out that the condition $g(a) = g(b) = 0$ is redundant. To see this, let $\tilde g$ be an absolutely continuous
function on $[a,b]$ with $\tilde g'\in L^2_p(a,b)$ and $\tilde gw > 0$, choose
a function $h\in\AC[a,b]$ such that
\begin{equation}\label{e:h}
h(x) > 0 \text{ for all }x\in (a,b),\quad h(a) = h(b) = 0,\quad \text{and}\quad h'\in L^2_p(a,b),
\end{equation}
and set $g := h\tilde g$. Then
$$g\in\AC[a,b],\quad g(a) = g(b) = 0,\quad gw > 0\quad\text{and}\quad g'= h'\tilde g + h\tilde g'\in L^2_p(a,b).$$
We note that a function $h$ with the above mentioned properties can be defined as follows:
Choose $x_0\in (a,b)$ such that
$\int_a^{x_0}1/\sqrt p = \int_{x_0}^b1/\sqrt p$ and let
$$
h(x) := \int_a^x\frac{\sgn(x_0-t)}{\sqrt{p(t)}}\,dt,\qquad x\in [a,b].
$$

We also mention that the condition $g(a) = g(b) = 0$ is not redundant for the eigenvalue estimates in Theorems \ref{t:main} and \ref{t:main2}.
\end{rem}

\vspace{.5cm}
\section{Bounds on exceptional real eigenvalues}

Let $A(\calD)$ be a selfadjoint realization of the indefinite Sturm-Liouville expression $\tau$ defined on some selfadjoint
domain $\calD$. It is well known that the resolvent of $A(\calD)$
is a compact operator and that the real eigenvalues of $A(\calD)$ accumulate to $+\infty$ and $-\infty$. Moreover the
real eigenvalues have the following {\it sign properties}; cf. \cite{cl}.

\begin{prop}
Let $\calD$ be a selfadjoint domain. Then there exist at most finitely many real eigenvalues $\lambda\neq 0$ of $A(\calD)$ with a corresponding eigenfunction $\phi$ such that
\begin{equation}\label{arthur}
\lambda[\phi,\phi]=\lambda \int_a^b \vert\phi\vert^2 w\,\,\leq 0.
\end{equation}
These eigenvalues will be called {\em real exceptional eigenvalues} of $A(\calD)$.
\end{prop}

We mention that $\la = 0$ is said to be an exceptional eigenvalue of $A(\calD)$ if there exists a function $\psi$ in the root subspace of $A(\calD)$ corresponding to zero such that $[A(\calD)\psi,\psi] < 0$. Furthermore, we note that in \cite{m2} the eigenfunctions corresponding to (non-zero) real exceptional eigenvalues satisfying \eqref{arthur} were called {\em real ghost states}.

In what follows we provide estimates on the real exceptional eigenvalues along the lines of Theorem~\ref{t:main}. The following preparatory lemma is the analog of Lemma~\ref{lemest1} and is also proved in Section \ref{s:lemmas}.

\begin{lem}\label{lemest3}
For all $\lambda\in\R\setminus\{0\}$ and all solutions $\phi\in\calD$ of the equation \eqref{ode} which satisfy \eqref{arthur}
we have
$$
\Vert\phi^\prime\Vert_{p,2}\leq \beta \Vert\phi\Vert_{\frac{1}{p},2}
\qquad\text{and}\qquad
\Vert\phi\Vert_\infty\leq \gamma\,\Vert\phi\Vert_{\frac{1}{p},2}.
$$
\end{lem}

Lemma~\ref{lemest3} implies the following variant of Theorem~\ref{t:main}; its proof remains the same. We leave it to the reader to formulate a variant of Corollary~\ref{c:ftp} for real exceptional eigenvalues.

\begin{thm}\label{t:main3}
Let $\calD$ be a selfadjoint domain, let $\alpha$, $\beta$, $\gamma$ be as above, and assume that there exists a real-valued function $g\in\AC[a,b]$ such that $gw > 0$ a.e.\ on $(a,b)$, $g(a) = g(b) = 0$, and $g'\in L^2_p(a,b)$. Then, with $\veps > 0$ chosen such that
$$
\mu_{\frac{1}{p}}\big(\{x\in [a,b] : p(x)g(x)w(x)<\veps\}\big)\,\le\,\frac{1}{2\gamma^2}\,,
$$
the following holds for all real exceptional eigenvalues $\lambda$ of the operator $A(\calD)$:
$$
|\la|\,\le\,\frac{2}{\veps}\,\bigl(\beta\,\gamma\,\|g'\|_{p,2} +  (\beta^2 + \gamma^2\|q\|_1)\|g\|_\infty\bigr).
$$
\end{thm}

\vspace{.5cm}
\section{Proofs of Lemmas~\ref{lemest1}, \ref{lemest2}, and \ref{lemest3}}\label{s:lemmas}
In this section we provide the remaining proofs of Lemma~\ref{lemest1}, Lemmma~\ref{lemest2}, and Lemma~\ref{lemest3}.

\begin{proof}[Proof of Lemma~\ref{lemest1}]
Choosing $x=a$ in \eqref{e:im} and taking into account \eqref{e:sa} and $\Im\la\not=0$ we find
\begin{equation}\label{e:known}
\int_a^b w|\phi|^2 = 0.
\end{equation}
From \eqref{e:integration} we then obtain
\begin{equation}\label{edinburgh}
 \begin{split}
  \Vert\phi^\prime\Vert^2_{p,2}&=(p\phi')(b)\ol{\phi(b)} - (p\phi')(a)\ol{\phi(a)} - \int_a^b q|\phi|^2\\
 &\leq c(\calD)\max\{|\phi(a)|^2,|\phi(b)|^2\}+\int_a^b \vert q_-\vert |\phi|^2 \\
 &\leq\bigl(c(\calD)+\Vert q_-\Vert_1\bigr)\Vert\phi\Vert_\infty^2 = \alpha \Vert\phi\Vert_\infty^2.
 \end{split}
\end{equation}
For $x,y\in [a,b]$, $y<x$,  we have
\begin{align*}
|\phi(x)|^2 - |\phi(y)|^2
&= \int_y^x \left(|\phi|^2\right)' = \int_y^x (\phi^\prime \ol\phi + \phi\ol\phi^\prime)\\
&\le 2\int_a^b|\phi'\phi| = 2\int_a^b \vert \phi'\vert p^{1/2}\,\vert\phi\vert p^{-1/2}\,\le\,2\Vert \phi'\Vert_{p,2}\Vert \phi\Vert_{\frac{1}{p},2},
\end{align*}
where we have used the Cauchy-Schwarz inequality in the last estimate. Multiplying the above inequality  with $p^{-1}(y)$ and integrating
over $[a,b]$ with respect to $y$ gives
$$
\vert\phi(x)\vert^2\Vert p^{-1}\Vert_1-\Vert \phi\Vert_{\frac{1}{p},2}^2\leq 2 \Vert \phi'\Vert_{p,2}\Vert \phi\Vert_{\frac{1}{p},2}\Vert p^{-1}\Vert_1
$$
for all $x\in[a,b]$.
Hence it follows that
\begin{equation}\label{e:first_est}
\|\phi\|_\infty^2 \,\le\, 2\|\phi'\|_{p,2} \Vert \phi\Vert_{\frac{1}{p},2}+  \Vert p^{-1}\Vert_1^{-1} \Vert \phi\Vert_{\frac{1}{p},2}^2.
\end{equation}
Therefore we obtain from \eqref{edinburgh}
$$
\Vert\phi^\prime\Vert^2_{p,2}\leq\alpha \Vert\phi\Vert_\infty^2\leq 2\alpha\|\phi'\|_{p,2} \Vert \phi\Vert_{\frac{1}{p},2}+ \alpha\,
\Vert p^{-1}\Vert_1^{-1} \Vert \phi\Vert_{\frac{1}{p},2}^2.
$$
This yields
$$
\bigl(\Vert\phi^\prime\Vert_{p,2}-\alpha\Vert \phi\Vert_{\frac{1}{p},2} \bigr)^2\leq \alpha\,
\Vert p^{-1}\Vert_1^{-1} \Vert \phi\Vert_{\frac{1}{p},2}^2  + \alpha^2\Vert \phi\Vert_{\frac{1}{p},2}^2
=\alpha\bigl(\Vert p^{-1}\Vert_1^{-1}+\alpha\bigr)\Vert \phi\Vert_{\frac{1}{p},2}^2
$$
and hence
\begin{equation}\label{masson}
\Vert\phi^\prime\Vert_{p,2}\leq \sqrt{\alpha\bigl(1 / \Vert p^{-1}\Vert_1 +\alpha\bigr)}\,\Vert \phi\Vert_{\frac{1}{p},2}+\alpha\Vert \phi\Vert_{\frac{1}{p},2}
=\beta \Vert \phi\Vert_{\frac{1}{p},2},
\end{equation}
so that the first estimate in the lemma is proved. The second estimate follows from the first one and \eqref{e:first_est}. Indeed, with the help
of \eqref{masson} we obtain from \eqref{e:first_est} that
$$
\|\phi\|_\infty^2\leq 2 \beta \Vert \phi\Vert_{\frac{1}{p},2}^2+  \Vert p^{-1}\Vert_1^{-1} \Vert \phi\Vert_{\frac{1}{p},2}^2
$$
holds, which implies $\|\phi\|_\infty\leq \sqrt{2\beta + 1/\Vert p^{-1}\Vert_1} \,\Vert \phi\Vert_{\frac{1}{p},2}=\gamma\,\Vert \phi\Vert_{\frac{1}{p},2}$.
\end{proof}

\begin{proof}[Proof of Lemma~\ref{lemest2}]
Let $W(x):=\int_a^x w$, $x\in[a,b]$,  and observe that integration by parts yields
$$\int_a^b W\,(\vert\phi\vert^2)^\prime =W(b)\vert\phi(b)\vert^2 - \int_a^b w\vert\phi\vert^2 = W(b)\vert\phi(b)\vert^2,$$
where we have used \eqref{e:known} in the last step. This implies
\begin{equation}
 \begin{split}
  W(b)\vert\phi(x)\vert^2 & = -W(b)\bigl(\vert\phi(b)\vert^2-\vert\phi(x)\vert^2\bigr)+\int_a^b W\,(\vert\phi\vert^2)^\prime\\
 & = -W(b)\int_x^b(\vert\phi\vert^2)^\prime+\int_a^b W\,(\vert\phi\vert^2)^\prime
 \end{split}
\end{equation}
and hence with $\Vert W\Vert_\infty\leq \Vert w\Vert_1$ and $W(b)=\int_a^b w$ we conclude that
 \begin{equation}
 \begin{split}
  \vert\phi(x)\vert^2 & = -\int_x^b(\vert\phi\vert^2)^\prime+\frac{1}{W(b)}\int_a^b W\,(\vert\phi\vert^2)^\prime\\
 & \leq 2 \int_a^b \vert\phi^\prime\phi\vert+2\frac{\Vert W\Vert_\infty}{\vert W(b)\vert }\int_a^b \vert\phi^\prime\phi\vert\\
 & \leq \left(2+2\frac{\Vert w\Vert_1}{\bigl| \int_a^b w\bigr| }\right)\int_a^b \vert\phi^\prime\vert p^{1/2} \vert \phi \vert p^{-1/2}\\
 & \leq \delta\, \Vert \phi^\prime\Vert_{p,2}\Vert \phi\Vert_{\frac{1}{p},2}
 \end{split}
\end{equation}
holds for all $x\in[a,b]$.
This leads to the estimate
 \begin{equation}\label{icms}
\Vert\phi\Vert^2_\infty\leq \delta\, \Vert \phi^\prime\Vert_{p,2}\Vert \phi\Vert_{\frac{1}{p},2}.
\end{equation}
As in the proof of Lemma~\ref{lemest1} we have $\Vert\phi^\prime\Vert^2_{p,2}\leq\alpha\Vert\phi\Vert^2_\infty$ (cf. \eqref{edinburgh})
which together with \eqref{icms} yields
$$\Vert\phi^\prime\Vert_{p,2}\leq\alpha\, \delta\, \Vert \phi\Vert_{\frac{1}{p},2}.$$
Plugging this into \eqref{icms} gives
$$\Vert\phi\Vert_\infty\leq\sqrt{\alpha}\,\delta\,\Vert \phi\Vert_{\frac{1}{p},2}$$
which completes the proof of Lemma~\ref{lemest2}.
\end{proof}

\begin{proof}[Proof of Lemma~\ref{lemest3}]
For a solution $\phi\in\calD$ we have
\begin{equation}
\la \int_a^b w|\phi|^2 = (p\phi')(a)\ol{\phi(a)} - (p\phi')(b)\ol{\phi(b)} + \int_a^b\big(p|\phi'|^2 + q|\phi|^2\big);
\end{equation}
cf. \eqref{e:integration} with $x=a$. The assumption \eqref{arthur} then implies the estimate
$$
\Vert\phi^\prime\Vert^2_{p,2}\leq (p\phi')(b)\ol{\phi(b)} - (p\phi')(a)\ol{\phi(a)} - \int_a^b q|\phi|^2
$$
and hence the estimate $\Vert\phi^\prime\Vert^2_{p,2}\leq\alpha\Vert\phi\Vert_\infty^2$ in \eqref{edinburgh} remains valid.
Thus the rest of the proof of Lemma~\ref{lemest1} holds also under the present assumptions on $\phi$ and yields the
estimates for $\Vert\phi^\prime\Vert_{p,2}$ and $\Vert\phi\Vert_\infty$.
\end{proof}


\begin{thebibliography}{99}


\bibitem{B85} R.\ Beals, Indefinite Sturm-Liouville problems and half-range completeness, J. Differential Equations 56 (1985), 391--407.

\bibitem{bpt}  J.\ Behrndt, F.\ Philipp, and C.\ Trunk,
               Bounds on the non-real spectrum of differential operators with indefinite weights,
               to appear in Math.\ Ann.

\bibitem{BT07}  J.\ Behrndt and C.\ Trunk, On the negative squares of indefinite Sturm-Liouville operators,
        J. Differential Equations 238 (2007), 491--519.

\bibitem{BBW02} P.\ Binding, P.\ Browne, and B.\ Watson, Spectral asymptotics for Sturm-Liouville equations with indefinite weight,
Trans. Amer. Math. Soc. 354 (2002), 4043--4065.

\bibitem{BLM04} P.~Binding, H. Langer, and M. M\"{o}ller, Oscillation results for indefinite Sturm-Liouville
problems with an indefinite weight function, J. Comput. Appl. Math. 171 (2004), 93--101.

\bibitem{BM10} P.\ Binding and M.\ M\"{o}ller, Negativity indices for definite and indefinite Sturm-Liouville problems,
 Math. Nachr. 283 (2010), 180--192.

\bibitem{BV01} P. Binding and H. Volkmer, Oscillation theory for Sturm-Liouville problems with indefinite coefficients,
Proc. Royal Soc. Edinburgh Sect. A 131 (2001), 989--1002.


\bibitem{cl}   B.\ \'Curgus and H.\ Langer,
               A Krein space approach to symmetric ordinary differential operators with an indefinite weight function,
               J.\ Differential Equations 79 (1989), 31--61.

\bibitem{h}    O.\ Haupt,
               \"{U}ber eine Methode zum Beweise von Oszillationstheoremen,
						   Math.\ Ann.\ 76 (1915), 67--104.

\bibitem{KKM09}  I.M.\ Karabash, A.S.\ Kostenko, and M.M.\ Malamud,
               The similarity problem for $J$-nonnegative Sturm-Liouville operators,
               J.\ Differential Equations 246 (2009), 964--997.

\bibitem{KMWZ03} Q.~Kong, M.~M\"{o}ller, H.~Wu, and A.~Zettl,
                Indefinite Sturm-Liouville problems,
                Proc.\ Roy.\ Soc.\ Edinburgh Sect.\ A 133 (2003), 639--652.




\bibitem{m2}   A.B.\ Mingarelli,
               A survey of the regular weighted Sturm-Liouville problem - The non-definite case,
                             arXiv:1106.6013v1.

\bibitem{p}    F.\ Philipp,
               Indefinite Sturm-Liouville operators with periodic coefficients,
               to appear in Oper.\ Matrices.

\bibitem{r}    R.G.D.\ Richardson,
               Contributions to the study of oscillation properties of the solutions of linear differential equations of the second order,
               Amer.\ J.\ Math.\ 40 (1918), 283--316.

\bibitem{qc}   J.\ Qi and S.\ Chen,
               A priori bounds and existence of non-real eigenvalues of indefinite Sturm-Liouville problems,
               to appear in J.\ Spectral Theory.



\bibitem{z}    A.\ Zettl,
               Sturm-Liouville Theory,
               AMS, Providence, RI, 2005.
\end{thebibliography}
\end{document}